\documentclass[10pt]{amsart}
\usepackage[pagebackref]{hyperref}

\newcommand{\bbC}{{\mathbb C}}

\newcommand{\bbP}{{\mathbb P}}
\newcommand{\bbQ}{{\mathbb Q}}
\newcommand{\bbR}{{\mathbb R}}

\newcommand{\GL}{\operatorname{GL}}

\newcommand{\SL}{\operatorname{SL}}
\newcommand{\SO}{\operatorname{SO}}

\newcommand{\Sp}{\operatorname{Sp}}

\newcommand{\la}{\langle}
\newcommand{\ra}{\rangle}

\newcommand{\w}{{\mathchoice{\,{\scriptstyle\wedge}\,}{{\scriptstyle\wedge}}
      {{\scriptscriptstyle\wedge}}{{\scriptscriptstyle\wedge}}}}

\newcommand{\be}{\begin{equation}}
\newcommand{\ee}{\end{equation}}
\newcommand{\bpm}{\begin{pmatrix}}
\newcommand{\epm}{\end{pmatrix}}

\numberwithin{equation}{section}

\newtheorem{proposition}{Proposition}
\newtheorem{corollary}{Corollary}

\theoremstyle{remark}

\newtheorem{remark}{Remark}
\newtheorem{example}{Example}

\begin{document}

\author[R. Bryant]{Robert L. Bryant}
\address{Duke University Mathematics Department\\
         PO Box 90320\\
         Durham, NC 27708-0320}
\email{\href{mailto:bryant@math.duke.edu}{bryant@math.duke.edu}}
\urladdr{\href{http://www.math.duke.edu/~bryant}%
         {http://www.math.duke.edu/\lower3pt\hbox{\symbol{'176}}bryant}}

\title[Projective, contact, and null curves]
      {Notes on projective, contact, and null curves}

\date{December 16, 2017}

\begin{abstract}
These are notes on some algebraic geometry 
of complex projective curves,
together with an application to studying the contact
curves in~$\bbP^3$ and the null curves 
in the complex quadric~$\bbQ^3\subset\bbP^4$,
related by the well-known Klein correspondence.
Most of this note consists of recounting the classical background. 
The main application is the explicit classification
of rational null curves of low degree in~$\bbQ^3$.

I have recently received a number of requests for these notes,
so I am posting them to make them generally available.

\end{abstract}



\thanks{
Thanks to Duke University for its support via a research grant,
during which these notes were written. 
\hfill\break
\hspace*{\parindent}
}

\maketitle

\setcounter{tocdepth}{2}
\tableofcontents

\section{Introduction}\label{sec: intro}

These are notes containing 
the details of the proof of my claim~\cite{Bryant3} 
that there are no unbranched rational null curves~$\gamma:\bbP^1\to\bbC^3$ 
with simple poles and having total degree~$5$ or~$7$.
Claims to the contrary that have been made in the literature
(c.f., \cite{PandX1}) are in error.

Along the way, I explain a few other results of interest. 
Mostly these are consequences of the results in~\cite{Bryant3}.
Some of this material has, in the meantime, 
been rediscovered by others~\cite{BandW1, BandW2}.

For the convenience of the reader, 
I include some discussion of the algebraic geometry
of projective curves.  All of this material is classical~\cite{GandH}.

\section{Invariants of projective curves}
Let $V$ be a complex vector space of dimension $n{+}1\ge 2$,
and let $\bbP(V)$ be its projectivization.  
When $V$ is clear from context, 
I will write~$\bbP^n$ for~$\bbP(V)$.

Let $S$ be a connected Riemann surface 
and let $f:S\to\bbP^n = \bbP(\bbC^{n+1})$ 
be a \emph{nondegenerate} holomorphic curve, 
i.e., $f(S)$ does not lie in any proper hyperplane $H^{n-1}\subset\bbP^n$.

When $S$ is compact, the \emph{degree} of $f$, $\deg(f)$, 
is the number of points in the pre-image $f^{-1}(H)\subset S$
where $H\subset \bbP^n$ is any hyperplane that is nowhere tangent to~$f$.
When $f:S\to\bbP^n$ is nondegenerate, one knows that $\deg(f)\ge n$. 

\subsection{Ramification}
Given $p\in S$, one can write 
\be\label{eq: f-expansion}
f = \left[h_0\,v^0 + h_1\,v^1 + \cdots h_n\,v^n\right]
\ee
for some basis $v^0,\ldots, v^n$ of $V$ 
where the $h_i$ are meromorphic functions on~$S$ 
that satisfy 
\be\label{eq: h-orders-for-f}
0 = \nu_p(h_0) < \nu_p(h_1) < \cdots < \nu_n(h_n),
\ee
where $\nu_p(h_i)$ is the order of vanishing of~$h_i$ at~$p\in S$.
The numbers $a_i(p) = \nu_p(h_i)\ge i$ for $0\le i\le n$ 
depend only on $f$ and $p$, 
not on the choice of basis $v^i$ and meromorphic functions~$h_i$
satisfying~\eqref{eq: f-expansion} and \eqref{eq: h-orders-for-f}.

For all but a closed, discrete set of points $p\in S$, 
one will have $a_i(p) = i$ for $0\le i \le n$. 
It is useful to define, for $i\ge 1$, 
$$
r_i(p) = a_i(p) - a_{i-1}(p) - 1 \ge 0,
$$
which is known as the $i$-th \emph{ramification degree of~$f$ at~$p$}.  
When $f$ is not clear from context, I will write $r_i(p,f)$. 

Since $r_i(p)=0$ for $0\le i\le n$
for all but a closed, discrete set of points~$p\in S$, 
one can define the \emph{$i$-th ramification divisor} of~$f$ 
to be the locally finite formal sum
\be\label{eq: ram-divs}
R_i(f) = \sum_{p\in S}\ r_i(p,f)\,\cdot p\,.
\ee
When $S$ is compact, this is a finite sum, 
in which case, $R_i(f)$ is an effective divisor on~$S$. 

\begin{remark}[Branch points]
A point $p\in S$ at which $r_1(p,f)>0$ is said to be 
a \emph{branch point} of~$f$ of order $r_1(p,f)$.
When $R_1(f) = 0$, $f$ is said to be \emph{unbranched},
which is equivalent to $f$ being an immersion.
\end{remark} 

\subsection{The associated curves}
Since $f$ is nondegenerate, 
there is a well-defined sequence of \emph{associated curves},
$f_k:S\to\bbP\bigl(\Lambda^k(V)\bigr)$ for $1\le k \le n$,
defined, relative to any local holomorphic coordinate $z:U\to\bbC$
where $U\subset S$ is an open set, by
$$
f_k=
\left[F\wedge\frac{dF}{dz}\wedge\cdots\wedge\frac{d^{k-1}F}{dz^{k-1}}\right]
$$
where $F:U\to V$ is holomorphic and non-vanishing 
and $f = [F]$ on $U\subset S$.  (It is easy to show that $f_k$ 
is well-defined, independent of the choice of $z$ or $F$.)
Of course, $f_1 = f$.

\begin{remark}[Wronskians]\label{rem: Wronsk}
If $h_1,\ldots, h_k$ are meromorphic functions on a connected
Riemann surface~$S$ and $z:U\to\bbC$ is a local holomorphic 
coordinate on $U\subset S$, then the \emph{Wronskian differential}
of $(h_1,\ldots,h_k)$ is the expression
$$
W(h_1,\ldots,h_k) 
= \det\begin{pmatrix}
h_1& h_2 & \cdots & h_k\\
h_1^{(1)} & h_2^{(1)} & \cdots & h_k^{(1)}\\
\vdots & \vdots & \ddots & \vdots\\
h_1^{(k-1)} & h_2^{(k-1)} & \cdots & h_k^{(k-1)}
 \end{pmatrix} \mathrm{d}z^{k(k-1)/2},
$$
where $h^{(j)}_k = d^jh_k/dz^j$.
It is not hard to show that $W(h_1,\ldots,h_k)$ does not depend
on the choice of local holomorphic coordinate $z$ and hence is
a globally defined (symmetric) differential on~$S$. 

The Wronskian has two important (and easily proved) properties that will
be needed in the rest of these notes.

First, if $\nu_p(h_1)< \nu_p(h_2) <\cdots < \nu_p(h_k)$, then 
$$
\nu_p\bigl(W(h_1,\ldots,h_k)\bigr) 
= \nu_p(h_1)+\cdots+\nu_p(h_k) - \tfrac12k(k{-}1).
$$

Second (and this follows easily from the first fact), 
$W(h_1,\ldots,h_k)$ vanishes identically if and only if 
the functions $h_1,\ldots,h_k$ are linearly dependent as functions on~$S$.
\end{remark}

Note that, when $f:S\to\bbP^n$ is described as in~\eqref{eq: f-expansion},
the associated curves can be written in the form
\be\label{eq: f_k-by-Wronsk}
f_k = \left[\sum_{0\le i_1 < i_2<\cdots< i_k\le n} 
W(h_{i_1},\ldots,h_{i_k})\,v^{i_1}{\w}v^{i_2}{\w}\cdots{\w}v^{i_k}\right].
\ee

\subsection{The canonical $k$-plane and line bundles}
  
Since $f_k(p)$ is the projectivization of a nonzero simple $k$-vector
for all $p\in S$, it follows that there exists a flag of subspaces
$$
\{0\}=E_0(p)\subset E_1(p)\subset\cdots\subset E_{n}(p)\subset E_{n+1}(p)=V
$$
such that $\dim E_i(p) = i$ 
and $f_i(p)= \bbP\bigl(\Lambda^i(E_i(p))\bigr)$
for all $p\in S$ and $i\ge1$.  

It is easy to show that the subset
\be
E_i = \{ (p,v)\in S\times V\ \vrule\ v\in E_i(p)\ \}
\ee
is a holomorphic $i$-plane subbundle of the trivial bundle~$E_{n+1}
=S\times V$. Since $E_{i-1}\subset E_{i}$, 
there are well-defined quotient line bundles over~$S$ 
\be\label{eq:  L_i-defined}
L_i = E_i/E_{i-1}
\ee
for $1\le i \le n{+}1$.  Note that, 
since $\Lambda^i(E_i)\simeq L_i\otimes \Lambda^{i-1}(E_{i-1})$
for $1\le i\le n{+}1$, it follows that 
\be\label{eq: trivial-tensor}
L_1\otimes L_2 \otimes \cdots \otimes L_{n+1} 
\simeq \Lambda^{n+1}(E_{n+1}) = S\times\bbC.
\ee 

Let $B \subset S\times V \times V\times\cdots\times V$ 
(with $n{+}1$ factors of~$V$) 
be the set of $(n{+}2)$-tuples~$(p,v_0,\ldots,v_n)$ 
that satisfy the conditions $p\in S$, $v_i\in E_{i+1}(p)$ 
for $0\le i\le n$, and $(v_0,\ldots,v_n)$ is a basis of~$V$.
This $B$ is a holomorphic submanifold of $S\times V^{n+1}$,
the projection $\sigma:B\to S$ onto the first factor 
is a submersion, and the $V$-valued functions $e_i:B\to V$ 
defined by~$e_i(p,v_0,\ldots,v_n) = v_i$ for $0\le i\le n$
are holomorphic.  Consequently, there are unique holomorphic
$1$-forms~$\omega^j_i$ on $B$ satisfying the structure equations
\be
\mathrm{d}e_i = e_j\,\omega^j_i,
\ee
and 
\be
\mathrm{d}\omega^j_i = -\omega^j_k\w\omega^k_i
\ee
Moreover, since, by construction,
\be
f_i\circ\sigma = [ e_0{\w}e_1{\w}\cdots{\w}e_{i-1}],
\ee
it follows that $\omega^j_i = 0$ whenever $j>i+1$
and that $\omega^{i+1}_i$ is $\sigma$-semibasic for $0\le i\le n$.

Now, for $b=(p,v_0,\ldots,v_n)\in B$, one has $v_i\in E_i(p)$
but $v_i\not\in E_{i-1}(p)$ for $i\ge1$.  
Consequently, there is a unique
linear function $\epsilon^i(b):E_i(p)\to\bbC$ that has $E_{i-1}(p)$
as its kernel and satisfies $\epsilon^i(v_i) = 1$.  Thus,
$\epsilon^i(b)$ can be regarded as a nonzero linear function 
on the line $E_i(p)/E_{i-1}(p)= L_i(p)$, and hence $\epsilon^i(b)$
is a nonzero element of the dual line $L_i(p)^*$.  In addition,
there is an element $[e_i](b)\in E_i(p)/E_{i-1}(p) = L_i(p)$ 
that is given by $[e_i(b)] = v_i\mod E_{i-1}(p)$.

With these definitions, it is not difficult to show 
that there is a well-defined section~$\rho_i$ 
of the line bundle $L_{i+1}\otimes L_i^*\otimes K$ over~$S$ 
(where $K$ is the canonical line bundle of~$S$) satisfying 
\be
\rho_i\circ\sigma = [e_{i+1}] \otimes \epsilon^i \otimes \omega^{i+1}_i\,.
\ee
Moreover, following the definitions above,
one finds that the section~$\rho_i$ 
vanishes to order $r_i(p)$ at $p\in S$.

\subsection{The compact case and divisors}
Now suppose that $S$ is compact, 
and fix a nondegenerate $f:S\to\bbP^n$, 
which will not be notated in the following discussion.

Then $L_i \simeq \mathcal{O}(-D_i)$ for $1\le i\le n$, 
where $D_i$ is a divisor on $S$, well-defined up to linear
equivalence.

From~\eqref{eq: trivial-tensor}, it then follows that
\be\label{eq: divs-sum-to-zero}
D_1+ D_2 + \cdots D_{n+1} \equiv 0,
\ee
where `$\equiv$' means linear equivalence of divisors.

Moreover, because the zero divisor of the holomorphic
section~$\rho_i$ of $L_{i+1}\otimes L_i^*\otimes K$ is $R_i$ 
for $1\le i \le n$, it follows that 
\be
R_i \equiv - D_{i+1} + D_i + K,
\ee 
where, again, $K$ is the canonical divisor of $S$.  
In particular, for $\ell>1$ we have
\be
D_{\ell} \equiv D_1 + (\ell{-}1)\,K - R_1 - R_2 -\cdots - R_{\ell-1}.
\ee

Moreover, using~\eqref{eq: divs-sum-to-zero}, 
one obtains
\be
(n{+}1)\,D_1+{{n{+}1}\choose 2}\,K 
   \equiv n\,R_1+(n{-}1)\,R_2+\cdots+R_n\,.
\ee

Since $\deg D_1  = \deg f$, taking degrees of divisors, one has 
\be\label{eq: deg-genus-ram}
(n{+}1)\,\deg f + n(n{+}1)(k{-}1) = n\,r_1+(n{-}1)\,r_2+\cdots+r_n\,
\ee
where $r_i= \deg R_i \ge 0$ and $k$ is the genus of~$S$. 

\begin{example}[Rational normal curves]
\label{ex: unramifiedcurves}
If $S$ is compact and $f:S\to\bbP^n$ is nondegenerate 
and satisfies $r_i=0$ for all~$i$, 
it follows from~\eqref{eq: deg-genus-ram} that $k=0$ 
and $\deg f = n$, so that $f(S)\subset\bbP^n$ 
is the rational normal curve of degree~$n$, i.e.,
up to projective equivalence,
\be\label{eq: ratl-normal-curve}
f = [1,z,z^2,\ldots,z^n]
\ee
where $z$ is a meromorphic function on~$S=\bbP^1$ 
with a single, simple pole.
\end{example}

To conclude this subsection, I list a few further useful facts.
First,
\be\label{eq: degree-of-f_i}
\deg f_i = \deg D_1  + \deg D_2  + \cdots + \deg D_i 
\ee
for $1\le i\le n$.

Next, the \emph{dual curve} 
$f_n:S\to\bbP\bigl(\Lambda^n(V)\bigr)=\bbP(V^*)$ 
of $f = f_1$ is nondegenerate, 
and its ramification divisors are given by
\be\label{eq: rams-of-dual-curve}
R_i(f_n) = R_{n+1-i}(f_1).
\ee
Moreover, the dual curve of $f_n$ is $f_1$, 
i.e., $(f_n)_n = f_1 = f$.

Finally, one has the following relation 
between the first ramification divisor of~$f_i$ 
and the $i$-th ramification divisor of~$f$:
\be\label{eq: branch-of-f_i}
R_1\bigl(f_i\bigr) = R_i(f).
\ee
(This follows immediately from~\eqref{eq: f_k-by-Wronsk} 
and the properties of the Wronskian.)
However, note that, in general, for $1 < i < n$,  
the higher ramification divisors of~$f_i$
cannot be computed solely in terms of the ramification divisors 
of~$f = f_1$.  In fact, the $f_i$ in this range 
need not even be nondegenerate, as will be seen.

\section{Contact curves in $\bbP^3$}

Now let $V$ have dimension~$4$ and let~$\beta\in\Lambda^2(V^*)$
be a nondegenerate $2$-form on~$V$,
i.e., $V$ is a symplectic vector space of dimension~$4$.
(Since any two nondegenerate $2$-forms on~$V$ are $\GL(V)$-equivalent,
the particular choice of~$\beta$ is not important.)
Let $\Sp(\beta)\subset \GL(V)$ denote the group of linear transformations
of~$V$ that preserve~$\beta$.  

The choice of~$\beta$ 
defines a volume form~$\Omega=\tfrac12\beta^2\in\Omega^4(V^*)$ 
on~$V$ and, because of the nondegenerate pairing
$$
\Lambda^2(V)\times \Lambda^2(V^*)\to\bbC,
$$
it also defines a subspace~$W = \beta^\perp\subset\Lambda^2(V)$ 
of dimension~$5$.

Moreover, by the usual reduction process induced 
by the $\bbC^*$-action of scalar multiplication on~$V$,
the projective space $\bbP^3 = \bbP(V)$ inherits a \emph{contact
structure}, i.e., a holomorphic $2$-plane field $C\subset T\bbP^3$
that is nowhere integrable and is invariant under the induced action 
of~$\Sp(\beta)$ on~$\bbP^3$.

A connected holomorphic curve~$f:S\to\bbP^3$ 
is said to be a \emph{contact curve}
with respect to~$\beta$ if $f'(T_pS)\subset C_{f(p)}\subset T_{f(p)}\bbP^3$
for all $p\in S$.  Equivalently, $f$ is a contact curve 
if and only if either $f$ is constant or else $f_2(S)$ 
has image in~$\bbP(W)\subset\bbP\bigl(\Lambda^2(V)\bigr)$.
If $f(S)$ does not lie in a line in~$\bbP^3$, I will say that $f$
is \emph{nonlinear}.

\begin{proposition}\label{prop: f_to_f2}
If $f:S\to\bbP^3$ is a nonlinear contact curve,
then $f$ is nondegenerate.  
Moreover, $R_1(f) = R_3(f)$,
and $f_2:S\to\bbP(W)\simeq\bbP^4$ is nondegenerate, with
$$
R_1(f_2) = R_4(f_2) = R_2(f)
\qquad\text{and}\qquad
R_2(f_2) = R_3(f_2) = R_1(f).
$$
\end{proposition}

\begin{proof}
If $f$ were degenerate, then $f(S)$ 
would be linearly full in some $\bbP^2\subset\bbP^3$, 
and hence it would be expressible
on a neighborhood of $p\in S$ in the form
$$
f = [\,v^0 + h_1\,v^1 + h_2\, v^2\,],
$$
where the $h_i$ are meromorphic functions on $S$ 
with $\nu_p(h_1) = a_1 > 0$ and $\nu_p(h_2) = a_2 > a_1$,
and with $v^0,v^1,v^2$ being linearly independent vectors in~$V$.  
If $z:U\to\bbC$ is a $p$-centered local holomorphic coordinate 
on an open $p$-neighborhood~$U\subset S$, 
and we set $\mathrm{d}h_i = h_i'\,\mathrm{d}z$,
then
$$
f_2 = [\,h_1'\,v^0\w v^1+h_2'\,v^0\w v^2+(h_1h_2'{-}h_2h_1')\,v^1\w v^2\,],
$$
where $\nu_p(h_1') = a_1{-}1 < \nu_p(h_2') = a_2{-}1 <
\nu_p(h_1h_2'{-}h_2h_1') = a_2{+}a_1{-}1$.  
It follows that the functions $h_1'$, $h_2'$, 
and $h_1h_2'{-}h_2h_1'$ are linearly independent on $U$,
implying that $f_2(S)\subset\bbP(W)$ is linearly full
in the projectivization of the span of~$\{v^0,v^1,v^2\}$, 
and hence 
that all of the $2$-vectors $\{v^0\w v^1,\,v^0\w v^2,\, v^1\w v^2\}$
must lie in~$W$.  However, this implies that $\beta$ 
vanishes on the $3$-plane spanned by~$\{v^0,v^1,v^2\}$, 
which is impossible, since $\beta$ is nondegenerate.
Thus, $f$ must be nondegenerate.    

Fix $p\in S$ and suppose that $f(p) = [v^0]$.  
Then one can choose $v^1$, $v^2$, $v^3$ in $V$ 
such that $(v^0,v^1,v^2,v^3)$ is a basis of $V$ for which
$$
\beta = \xi_0\w \xi_3 + \xi_1\w\xi_2,
$$
where $(\xi_0,\xi_1,\xi_2,\xi_3)$ is the dual basis of $V^*$ 
corresponding to $(v^0,v^1,v^2,v^3)$.  

Write
\be\label{eq: f-in-hv-basis}
f = [\,v^0 + h_1\, v^1 + h_2\, v^2 + h_3\,v^3\,]
\ee
for some meromorphic functions $h_i$ on~$S$ that vanish at~$p$ 
and select a local $p$-centered holomorphic coordinate $z:U\to \bbC$
on some $p$-neighborhood $U\subset S$. 
The condition that $f$ be contact with respect to $\beta$ 
is expressed as the equation 
\be\label{eq: h-ode}
h'_3 = h_2 h'_1 - h_1 h'_2
\ee
where $\mathrm{d}h_i = h'_i\,\mathrm{d}z$ on $U$.  

Since $f$ is nondegenerate, $h_2 h'_1 - h_1 h'_2$ 
does not vanish identically.
Hence, by making a change of basis in $(v^1,v^2)$, 
it can be assumed that $0<\nu_p(h_1)<\nu_p(h_2)$,
which, by~\eqref{eq: h-ode} and the fact that $\nu_p(h_3)>0$, 
forces 
$$
\nu_p(h_3)=\nu_p(h_1)+\nu_p(h_2).
$$
Set $a_i = \nu_p(h_i)$, so that $a_3 = a_1+a_2$ and $a_2>a_1>0$.

First, note that this implies that 
$$
r_3(p,f) = a_3{-}a_2-1 = a_1{-1} = r_1(p,f).
$$
Since this holds for all $p\in S$, it follows that $R_3(f) = R_1(f)$.

Second, the relation~\eqref{eq: h-ode} implies that
\be
\begin{aligned}
f_2 &= [\, h'_1\, v^0{\w}v^1 +  h'_2\, v^0{\w}v^1 
      +  h'_3\, (v^0{\w}v^3 {-} v^1{\w}v^2) \\
    &\qquad +  (h'_1h_3{-}h_3h'_1)\,v^1{\w}v^3 
            + (h'_2h_3{-}h'_3h_2)\,v^2{\w}v^3\,].
\end{aligned}
\ee
Now, the sequence of orders of vanishing of these five 
coefficients of the basis elements of~$W$ are five \emph{distinct} numbers:
$$
a_1{-}1 < a_2{-}1 < a_1{+}a_2{-}1 < 2a_1{+}a_2{-}1 < a_1{+}2a_2{-}1.
$$
Hence, $f_2:S\to\bbP(W)\simeq\bbP^4$ is nondegenerate 
and has the following ramification indices at $p$:
$$
\begin{aligned}
r_1(p,f_2) &= a_2{-}a_1{-1} = r_2(p,f),\\
r_2(p,f_2) &= (a_1{+}a_2){-}a_2{-1} = r_1(p,f),\\
r_3(p,f_2) &= (2a_1{+}a_2){-}(a_1{+}a_2){-1} = r_1(p,f),\\
r_4(p,f_2) &= (a_1{+}2a_2){-}(2a_1{+}a_2){-1} = r_2(p,f).
\end{aligned}
$$
Thus, $R_1(f_2)= R_4(f_2) = R_2(f)$ and $R_2(f_2)=R_3(f_2)=R_1(f)$,
as claimed.
\end{proof}

\begin{remark}
Proposition~\ref{prop: f_to_f2} was proved in~\cite{Bryant3},
though it was known classically~\cite{BandK}.
It also appears (in slightly different notation) in~\cite{BandW1},
the authors of which do not appear to have been aware of~\cite{Bryant3}.
\end{remark}

Proposition~\ref{prop: f_to_f2} also suggests a slightly more general notion 
of contact curve, which is described by the following result.

\begin{proposition}
Let $f:S\to\bbP(\bbC^4)\simeq\bbP^3$ be a nondegenerate holomorphic curve
for which $f_2:S\to\bbP(\Lambda^2(\bbC^4))\simeq\bbP^5$ is degenerate.
Then there exists a nondegenerate symplectic 
form~$\beta\in\Lambda^2((\bbC^4)^*)$, unique up to constant 
multiples, such that $f$ is a contact curve with respect to~$\beta$.
\end{proposition}

\begin{proof}
As before, fix a point~$p\in S$ and write $f$ in the form
\be\label{eq: f-in-hv-basis}
f = [\,v^0 + h_1\, v^1 + h_2\, v^2 + h_3\,v^3\,]
\ee
for some meromorphic functions $h_i$ on~$S$ that vanish at~$p$
and satisfy 
$$
0 < a_1 = \nu_p(h_1) < a_2 = \nu_p(h_2) < a_3 = \nu_p(h_3).
$$ 
Let $z$ be a meromorphic function on~$S$ 
that has a simple zero at $p$ 
and write $\mathrm{d}h_i = h_i'\,\mathrm{d}z$. 
Then $f_2$ takes the form
\be\label{eq: f_2-expansion}
\begin{aligned}
f_2 &= [\, h_1'\,v^0\w v^1 + h_2'\,v^0\w v^2 \\
&\qquad + h_3'\,v^0\w v^3 + (h_1h_2'{-}h_2h'_1)\,v^1{\w}v^2\\
& \qquad +  (h_1h_3'{-}h_3h'_1)\,v^1{\w}v^3 
            + (h_2h_3'{-}h_3h'_2)\,v^2{\w}v^3\,],
\end{aligned}
\ee
and the orders of vanishing at $p$ 
of the mermomorphic coefficients of these terms 
in the order written in~\eqref{eq: f_2-expansion} are
\be\label{eq: initial-ineqs}
a_1{-}1 < a_2{-}1 < 
\begin{matrix} a_3{-}1\\ a_{2}{+}a_1{-}1\end{matrix}
 < a_{3}{+}a_1{-}1 < a_{3}{+}a_2{-}1.
\ee
If $a_3$ were not equal to $a_2+a_1$, then these six integers 
would be distinct, and it would follow 
that the six coefficient functions were linearly
independent as meromorphic functions on~$S$.  
In this case, $f_2$ would be linearly full in $\Lambda^2(\bbC^4)$,
contrary to hypothesis.  Thus, we must have $a_3=a_2+a_1$, and
the inequalities~\eqref{eq: initial-ineqs} become
\be\label{eq: sharp-ineqs}
a_1{-}1 < a_2{-}1 < a_{2}{+}a_1{-}1 < a_{2}{+}2a_1{-}1 < 2a_{2}{+}a_1{-}1.
\ee

Now, in order for $f_2$ to be degenerate, these six coefficients
must satisfy at least one nontrivial linear relation 
with constant coefficients.
Because of the strict inequalities~\eqref{eq: sharp-ineqs}, 
this relation cannot involve $h_1'$ or $h_2'$, so it must of the form
$$
c_1\,h_3' + c_2\,(h_1h_2'{-}h_2h'_1) + c_3\,(h_1h_3'{-}h_3h'_1)
+c_4\,(h_2h_3'{-}h_3h'_1) = 0.
$$
Moreover, again because of the strict inequalities~\eqref{eq: sharp-ineqs}, 
neither $c_1$ nor $c_2$ can vanish, and, thus, there cannot
be two independent linear relations of this kind.

Now, consider the $2$-form
$$
\beta = c_1\,\xi_0{\w}\xi_3 + c_2\,\xi_1{\w}\xi_2
       + c_3\,\xi_1{\w}\xi_3 + c_4\,\xi_2{\w}\xi_3\,,
$$
where $(\xi_0,\xi_1,\xi_2,\xi_3)$ is the dual basis in $(\bbC^4)^*$
to the basis $(v^0,v^1,v^2,v^3)$ of $\bbC^4$.  Since
$\beta\w\beta = 2c_1c_2\,\xi_0{\w}\xi_1{\w}\xi_2{\w}\xi_3\not=0$,
the $2$-form $\beta$ is nondegenerate and hence defines a symplectic
structure on $\bbC^4$.  By construction, $f_2$ lies in the
projectivization of $W\subset\Lambda^2(\bbC^4)$, the kernel of~$\beta$.
Hence, $f$ is a contact curve in the projectivization of the
symplectic space $(\bbC^4,\beta)$.

Since there is only one linear relation among the meromorphic coefficients 
appearing in $f_2$, it follows that $f_2$ lies linearly fully 
in $\bbP(W)\simeq\bbP^4$, which proves the uniqueness of~$\beta$ up to multiples.  
\end{proof}

\begin{example}[Rational contact curves of arbitrary degree]
\label{ex: ratl-contact-of-all-degrees}
Let $p$ and $q$ be relatively prime integers satisfying $0<p<q$,
and consider the curve $f:\bbP^1\to\bbP^3$, where $z$ is a meromorphic
function on~$\bbP^1$ possessing a single, simple pole at~$P$ and
a single, simple zero at~$Q$, defined by
$$
f = [\,v^0 + z^p\,v^1 + z^q\,v^2 + z^{p+q}\,v^3\,],
$$
where $v^0,v^1,v^2,v^3$ are linearly independent in~$\bbC^4$.
Computation yields
$$
\begin{aligned}
f_2 &= [\,p\,v^0{\w}v^1 + qz^{q-p}\,v^0{\w}v^2 
 + z^q\bigl((p{+}q)\,v^0{\w}v^3+(q{-}p)\,v^1{\w}v^2\bigr)\\
&\qquad{}+qz^{p+q}\,v^1{\w}v^3 + pz^{2q}\,v^2{\w}v^3\,]. 
\end{aligned}   
$$
Thus, $f$ is a contact curve for the symplectic structure
$$
\beta = (p{-}q)\,\xi_0{\w}\xi_3 + (p{+}q)\,\xi_1{\w}\xi_2\,,
$$
and one has $R_1(f) = (p{-}1)(P+Q)$, while $R_2(f)= (q{-}p{-}1)(P+Q)$.
This example, for $q=p{+}1$, appears in~\cite{Bryant3}.
\end{example}

\section{Null curves in $\bbC^3$ and $\bbQ^3$}

Endow $\mathbb{C}^3$ with a nondegenerate (complex) inner product,
which will be denoted $v{\cdot}w\in\bbC$ for $v,w\in\bbC^3$.

If $S$ is a connected Riemann surface, 
then a non-constant meromorphic curve~$\gamma:S\to\bbC^3$ 
will be said to be a \emph{null curve} 
if the meromorphic symmetric quadratic form 
$\mathrm{d}\gamma\cdot\mathrm{d}\gamma$ vanishes identically on~$S$. 

In order to treat the poles of meromorphic null curves algebraically, 
it will be useful to introduce an algebraic compactification of~$\bbC^3$.
The usual compactification that regards~$\bbC^3$ as an affine
open set in $\bbP^3$ is not useful in this context, 
since there is no natural way to extend the notion of `null' 
to the hyperplane at infinity.

Instead, one embeds $\bbC^3$ into $\bbP^4$ as a quadric hypersurface 
$Q_3$ by identifying $x\in\bbC^3$ 
with the point 
$$
[1,\,x,\, x{\cdot}x] = [1,\,x_1,\,x_2,\,x_3,\,{x_1}^2{+}{x_2}^2{+}{x_3}^2]
\in\bbP^4 = \bbP(\bbC^5).
$$
The resulting image is an affine chart on the projective quadric
$\bbQ^3\subset\bbP^4$ defined by the homogeneous equations
\be\label{eq: nullquadform}
X_0X_4 - {X_1}^2 - {X_2}^2 - {X_3}^2  = 0.
\ee 

A meromorphic null curve~$\gamma:S\to\bbC^3$
completes uniquely to an algebraic curve $g:S\to\bbQ^3\subset\bbP^4$
that is also null, in the sense that the tangent lines to the curve
lie in~$\bbQ^3$ as well.

Moreover, \eqref{eq: nullquadform} is the
quadratic form associated to an inner product $\la,\ra$ on~$\bbC^5$
with the property that a $g:S\to\bbQ^3$ that is the completion of
a meromorphic null curve~$\gamma:S\to\bbC^3$ is of the form $g = [G]$
where $G:S\to\bbC^5$ is meromorphic and satisfies 
\be\label{eq: Gnullcondition}
\la G, G\ra = \la\mathrm G,\mathrm{d}G\ra 
= \la\mathrm{d}G,\mathrm{d}G\ra = 0.
\ee
(In the last equation, $\la\mathrm{d}G,\mathrm{d}G\ra$ is
to be interpreted as a \emph{symmetric} meromorphic quadratic form.)

\begin{proposition}
If $\gamma:S\to\bbC^3$ is a meromorphic null curve 
with $d$ simple poles and no other poles, then the completed null
curve~$g:S\to\bbQ^3$ has degree $d$ $($as a map to $\bbP^4$$)$.  
If, in addition, $\gamma$ is an immersion away from its poles 
$($i.e., $\gamma$ is unbranched$)$, 
then $g:S\to\bbQ^3$ is also an immersion.
\end{proposition}

\begin{proof}
This follows immediately from local computation.
\end{proof}

\subsection{The Klein correspondence}
I now recall the famous \emph{Klein correspondence}
between nondegenerate contact curves $f:S\to\bbP^3$ 
and nondegenerate null curves $g:S\to\bbQ^3\subset\bbP^4$.

As before, let $V$ be a symplectic complex vector space of dimension~$4$
with symplectic form~$\beta\in\Lambda^2(V^*)$, 
with $\Omega = \tfrac12\beta^2\in\Lambda^4(V^*)$ a volume form on~$V$.
Let $W\subset \Lambda^2(V)$ be the $5$-dimensional subspace 
annihilated by~$\beta$.  
Then there is a nondegenerate symmetric inner product $\la,\ra$ on~$W$
defined by
$$
\la w_1, w_2\ra = \Omega(w_1\wedge w_2).
$$
The (connected) symplectic group $\Sp(\beta)\subset\GL(V)$ 
acts on $\Lambda^2(V)$ preserving~$W$
and preserving this inner product.   Morover $g(w) = w$
for all $w\in W$ if and only if $g = \pm I_V$, thus defining
a double cover $\Sp(\beta)\to\SO\bigl(\la,\ra\bigr)$, which
is one of the so-called `exceptional isomorphisms'.

Note that $\la w, w\ra = 0$ for a nonzero $w\in W$ 
if and only if $w$ is a \emph{decomposable} $2$-vector, 
i.e., $w = v_1\w v_2$ for two linearly independent vectors $v_1,v_2\in V$.
Such an element $w$ will be said to be a \emph{null} vector in $W$.
Define
$$
\bbQ^3=\bigl\{\ [w]\in \bbP(W)\ \vrule\ \la w, w\ra = 0\ \bigr\}\subset\bbP(W).
$$
Then $\bbQ^3$ is the \emph{null hyperquadric} of $\la,\ra$.
Since $\la,\ra$ is nondegenerate, $\bbQ^3$ 
is a smooth hypersurface in~$\bbP(W)\simeq\bbP^4$.

If $f:S\to\bbP(V)$ is a nondegenerate contact curve, 
then $g = f_2$ has image in $W$ and, moreover, since, by construction, 
$g(p)$ is the projectivization of a decomposable $2$-vector
for all $p\in S$, it follows that $g(S)\subset\bbQ^3$.

In fact, more is true:  Writing $f = [F]$ where $F:S\to V$
is meromorphic and letting $z:U\to\bbC$ be a local holomorphic
coordinate on~$U\subset S$ and writing $\mathrm{d}H = H'\,\mathrm{d}z$
for any meromorphic $H$ on~$S$, one obtains $g = [G]$,
where $G = F\w F'$.  Since $G' = F\w F''$, it follows that
$$
\begin{aligned}
\la G,G\ra&=\la F\w F',F\w F'\ra=\Omega(F\w F'\w F\w F')=\Omega(0)=0\\
\la G,G'\ra&=\la F\w F',F\w F''\ra=\Omega(F\w F'\w F\w F'')=\Omega(0)=0\\
\la G',G'\ra&=\la F\w F'',F\w F''\ra=\Omega(F\w F''\w F\w F'')=\Omega(0)=0\\
\end{aligned}
$$
Hence, $g:S\to\bbQ^3$ is a null curve, which, 
by Proposition~\ref{prop: f_to_f2}, 
is nondegenerate as a curve in~$\bbP(W)$.
 
The interesting thing is the converse, which is due to Klein: 

\begin{proposition} 
If $g:S\to\bbQ^3$ 
is a null curve that is nonlinear, i.e., its image is not
contained in a linear $\bbP^1\subset\bbP(W)$, then $g=f_2$
for a unique nondegenerate contact curve~$f:S\to\bbP(V)$.
\end{proposition}

\begin{proof}
The result is local,  so write $g = [G]$
where $G:U\to W$ is holomorphic and nonvanishing.  
By hypothesis, $G\w G = 0$, which, of course, 
implies that $G\w G' = 0$.  The condition that $g:S\to \bbQ^3$
be null is then equivalent to $G'\w G' = 0$.  Thus, $G$ and $G'$
span a null $2$-plane in $W$.  Since $G$ and $G'$ are each
decomposable, while $G\w G' = 0$, it follows that they can be
written in the form $G = F \w H$ and $G' = F \w K$ 
for some meromorphic $F,H,K:U\to W$ that are virtually linearly independent,
i.e., $F\w H\w K$ vanishes only at isolated points.  
Since $G' = F'\w H + F\w H' = F \w K$, 
it follows that $F'\w H = F\w (K-H')$, so that $F\w F' \w H$ 
vanishes identically, implying that $F$, $F'$, and $H$ 
are linearly dependent and hence span a $2$-dimensional 
vector space.  

Now, I claim that $F\w F'$ does not vanish identically, 
which implies that $[F\w F'] = [F\w H] = g$, and hence that $g = f_2$
where $f = [F]$.  

Suppose, on the contrary, that $F\w F'$ did vanish identically.
In that case $F = h v^0$ for some vector~$v^0\in V$, 
unique up to multiples, and some meromorphic function~$h$.
Thus, we can assume that $F = v^0$, 
and so $G = v^0\w H$. Consequently, $G' = v^0\w H'$
and $G'' = v^0\w H''$.  If $v^0$, $H$, $H'$, and $H''$ 
were linearly independent on any open set, 
it would follow that the subspace of $W$ 
spanned by $G$, $G'$, and $G''$, would be of dimension $3$
and totally null (since all of the elements are decomposable), 
which is impossible since $\la,\ra$
is nondegenerate and $W$ has dimension~$5$. 
Thus, $H''$ is a linear combination
of $v^0$, $H$ and $H'$.  Consequently, the $3$-plane
spanned by $v^0$, $H$, and $H'$ is constant, implying
that the $2$-plane in $W$ spanned by $G$ and $G'$ is
constant, which implies that $g(S)\subset\bbQ^3$
is a line in $\bbP^4$, contrary to hypothesis.

It is now established that $g = f_2$ 
where $f = [F]:S\to\bbP(V)$ is a contact curve,
and the uniqueness of $f$ is clear.
Since $g$ is not constant, 
$f(S)$ does not lie in a line in $\bbP(V)$ 
and hence, by Proposition~\ref{prop: f_to_f2},
$g:S\to\bbQ^3$ is nondegenerate in~$\bbP(W)$.
\end{proof}

\subsection{Ramifications and degrees}
Now suppose that $S$ is a compact (connected) Riemann
surface and that $f:S\to\bbP(V)$ is a holomorphic contact
curve that is not contained in a line and that 
$g=f_2:S\to\bbQ^3\subset\bbP(W)$
is its Klein-corresponding null curve.  
The Pl\"ucker formula~\eqref{eq: deg-genus-ram}, 
coupled with the fact that $r_3(f) = r_1(f)$, 
implies that
\be\label{eq: contactPlucker1}
4\deg(f) + 12(k-1) = 4r_1(f) + 2r_2(f),
\ee
where $k$ is the genus of~$S$.
Note that, in consequence, $r_2(f)$ is always even.
Meanwhile, Proposition~\ref{prop: f_to_f2} 
and \eqref{eq: deg-genus-ram} imply
\be\label{eq: contactPlucker2}
5\deg(g) + 20(k-1) = 5r_1(f) + 5r_2(f).
\ee

\begin{example}\label{ex: unbranchedrationalnullcurves}
The case of most interest in these notes will be
when $g$ is unbranched, i.e., $r_2(f)=0$,
and $S$ has genus $k=0$, in which case, 
the formulae above reduce to 
\be\label{eq: unbranchedrationalnullcurves}
\deg(g) = \deg(f) +1
\qquad\text{and}\qquad
r_1(f) = \deg(f)-3.
\ee
These relations will be useful in the sequel.
\end{example}

\section{Rational null curves of low degrees}

With the above preliminaries out of the way, 
I can now provide an analysis of the possibilities when 
$f:\bbP^1\to\bbP^3$ is a rational contact curve of low degree
such that $f_2:\bbP^1\to\bbQ^3$ is unbranched.

A contact curve $f:\bbP^1\to\bbP(V)$
of degree $1$ is linear, and a null curve~$f:\bbP^1\to\bbQ^3$
is linear. These linear cases will be set aside from now on.

\begin{example}[Even degrees]
\label{ex: null-even-degree}
Explicit unbranched null curves~$g:\bbP^1\to\bbQ^3$
are provided by Example~\ref{ex: ratl-contact-of-all-degrees}.
The curve $g=f_2$ has even degree~$2p{+}2\ge4$
and is unbranched whenever~$q = p{+}1$.
\end{example}

However describing all the unbranched null curves in~$\bbQ^3$ 
of any given degree seems to be a harder problem.

\subsection{Degree at most $4$}
The very lowest possible degrees are easy to treat.

\begin{proposition}
If $f:\bbP^1\to\bbP(V)$ is a nonlinear contact curve 
of degree at most~$3$, then $f(\bbP^1)\subset\bbP(V)$ 
is a rational normal curve. 
All contact rational normal curves are symplectically equivalent.
\end{proposition}

\begin{proof}
By~\eqref{eq: contactPlucker1}, if $f:\bbP^1\to\bbP^3$
is a nondegenerate contact curve, then
\be\label{eq: ratlcontactPlucker}
\deg(f) = 3 + r_1(f) + \tfrac12 r_2(f).
\ee
Consequently, $\deg(f)\ge 3$,
with equality only when $r_1(f)=r_2(f)=0$. 
Since $r_3(f)=r_1(f)=0$, it follows that $f:\bbP^1\to\bbP^3$
is completely unramified and hence is a rational
normal curve of degree $3$ (see Example~\ref{ex: unramifiedcurves}).

Conversely, if $\deg(f)=3$, then, 
choosing a meromorphic function~$z$ on~$\bbP^1$ 
with exactly one simple pole, write
$$
f = [v^0 + z\,v^1 + z^2\, v^2 + z^3\,v^3]
$$ 
for some basis $(v^0,v^1,v^2,v^3)$ of $V\simeq\bbC^4$
with dual basis $(\xi_0,\xi_1,\xi_2,\xi_3)$ of $V^*$.
Then
$$
g = f_2 = \bigl[\,v^0{\w}v^1 + 2z\,v^0{\w}v^2 
                 + z^2\,(3v^0{\w}v^3+v^1{\w}v^2)
                 + 2z^3\,v^1{\w}v^3 + z^4\,v^2{\w}v^3\,].
$$
Thus, $g(\bbP^1)$ lies linearly fully 
in the projectivization of the kernel~$W\subset\Lambda^2(V)$
of the nondegenerate $2$-form 
$\beta = \xi_0\w\xi_3-3\xi_1\w\xi_2\in\Lambda^2(V^*)$. 

Thus, all contact rational normal curves 
are symplectically equivalent.
\end{proof}

\begin{corollary}
If $g:\bbP^1\to\bbQ^3$ is a nonlinear null curve 
of degree at most~$4$, then $\deg(g) = 4$ and $g = f_2$ 
where~$f:\bbP^1\to\bbP^3$ is a contact rational normal curve.
In particular, there are no nonlinear null curves in~$\bbQ^3$
of degree $2$ or $3$.
\end{corollary}

\begin{proof}
Write $g = f_2$, where $f:\bbP^1\to\bbP^3$ is a contact curve.  
Then $\deg(g) = 4 + r_1(f) + r_2(f)$,
so $\deg(g)\le 4$ implies that $\deg(g) = 4$ and $r_1(f)=r_2(f)=0$.
Thus, $f:\bbP^1\to\bbP^3$ is a rational normal curve.
\end{proof}

\subsection{Degree $5$}
To begin, I classify the nonlinear rational contact curves of degree~$4$.

\begin{proposition}\label{prop: f_of_deg_4}
Up to symplectic equivalence, there is only one
nonlinear contact curve $f:\bbP^1\to\bbP^3$ of degree~$4$.
It satisfies $R_1(f) = 0$ and $R_2(f) = p+q$ where $p,q\in\bbP^1$
are distinct.
\end{proposition}

\begin{proof}
Let $f:\bbP^1\to\bbP^3$ be a nonlinear contact curve of degree~$4$. 
Then by~\eqref{eq: ratlcontactPlucker},
 $\bigl(r_1(f),r_2(f)\bigr)$ is either $(1,0)$
or $(0,2)$, and $f$ can be written in the form
$$
f = \bigl[\,v^0 + z\,v^1 + z^2\, v^2 + z^3\,v^3 + z^4\,v^4\,\bigr]
$$
for five vectors $v^0,\ldots,v^4$ in $\bbC^4$ 
that satisfy one linear relation 
and where $z$ is a meromorphic function on~$\bbP^1$ 
that has a single pole.  
Since $f$ has degree~$4$, neither $v^0$ nor $v^4$ can be zero.

If $r_1(f)=1$ and $r_2(f) = 0$, 
then $f$ branches to order $1$ at a single point of~$\bbP^1$,
which can be taken to be the pole of $z$. 
This implies that $v^3$ must be a multiple of~$v^4$. 
Hence, by replacing~$z$ by~$z+c$ for an appropriate
constant~$c$, it can be assumed that~$v^3 = 0$. 
Thus, the vectors $v^0,v^1,v^2,v^4$ are a basis of~$\bbC^4$.  
However, when
$$
f = \bigl[\,v^0 + z\,v^1 + z^2\, v^2 + z^4\,v^4\,\bigr],
$$ 
one finds that 
$$
f_2 = \bigl[\,v^0{\w}v^1+2z\,v^0{\w}v^2 + z^2\,v^1{\w}v^2 + 4z^3\,v^0{\w}v^4
              + 3z^4\,v^1{\w}v^4 + 2z^5\,v^2{\w}v^4\,\bigr],
$$
so that $f_2(\bbP^1)$ is nondegenerate 
in~$\bbP\bigl(\Lambda^2(\bbC^4)\bigr)\simeq\bbP^5$.
Thus, such an $f$ cannot be a contact curve 
with respect to any symplectic structure on $\bbC^4$.

If $r_2(f)=2$, there are two possibilities, either $R_2(f) = 2{\cdot}p$
for some $p\in\bbP^1$ or else $R_2(f) = p+q$ where $p,q\in\bbP^1$
are distinct.  Moreover, $r_1(f)=0$, which implies, 
in particular, that $v^0\w v^1$ and $v^3 \w v^4$ are nonzero.

In the first subcase, assume, without loss of generality, 
that $p$ is the zero of~$z$.  Since 
$$
f_2 = \bigl[ v^0{\w}v^1+2z\,v^0{\w}v^2 + z^2\,(3\,v^0{\w}v^3+v^1{\w}v^2) 
+ \cdots \bigr],
$$
where the unwritten terms vanish to order $3$ or more at $z=0$, 
the assumption that $R_1(f_2) = R_2(f) = 2{\cdot}p$ 
implies that $v^0{\w}v^2$ and $3\,v^0{\w}v^3+v^1{\w}v^2$ 
are multiples of $v^0{\w}v^1$.
This implies that both $v^2$ and $v^3$ lie in the linear span of $v^0$
and $v^1$, which is impossible, since $v^0,v^1, v^4$ cannot span $\bbC^4$. 

Meanwhile, if $R_2(f) = p+q$, where $p,q\in\bbP^1$ are distinct,
then we can choose $z$ so that $p$ and $q$ are defined by $z=0$
and $z = \infty$.  Since
$$
f_2 = \bigl[ v^0{\w}v^1+2z\,v^0{\w}v^2 + \cdots 
            + 2z^5\,v^2{\w}v^4 + z^6\,v^3{\w}v^4\bigr],
$$
where the unwritten terms vanish to order at least $2$ at $z=0$
and have a pole of at most order $4$ at $z=\infty$, it follows
from $R_1(f_2) = R_2(f) = p+q$ that $v^0{\w}v^2$ is a multiple of $v^0{\w}v^1$
and $v^2{\w}v^4$ is a multiple of $v^3{\w}v^4$.  In particular,
$v^2$ must be both a linear combination of $v^0$ and $v^1$ 
and a linear combination of $v^3$ and $v^4$.  Now, this can
only happen if $v^2 = 0$, since $v^0,v^1,v^3,v^4$ 
must be a basis of $\bbC^4$.   Thus,
$$
f = \bigl[\,v^0 + z\,v^1 + z^3\, v^3 + z^4\,v^4\,\bigr],
$$
which implies
$$
f_2 = \bigl[\,v^0{\w}v^1+3z^2\,v^0{\w}v^3+z^3\,(2\,v^1{\w}v^3+4\,v^0{\w}v^4)
              + 3z^4\,v^1{\w}v^4 + z^6\,v^3{\w}v^4\,\bigr].
$$
Thus, $f_2$ is linearly full in $\bbP(W)\simeq\bbP^4$, 
where $W\subset\Lambda^2(\bbC^4)$ is the $5$-dimensional subspace 
annihilated by the symplectic form
$$
\beta = \xi_0\w\xi_4 - 2\,\xi_1\w\xi_3
$$ 
(where $(\xi_0,\xi_1,\xi_3,\xi_4)$ is the basis of $V^*\simeq\bbC^4$ 
that is dual to the basis $(v^0,v^1,v^3,v^4)$ of $\bbC^4$).

Thus, $f$ is a contact curve with respect to the contact structure 
on $\bbP^3$ defined by~$\beta$.  
The uniqueness of $f$ up to symplectic equivalence is now clear.
\end{proof}

\begin{corollary}
There is no nonlinear null curve~$g:\bbP^1\to\bbQ^3$ of degree~$5$.
\end{corollary}

\begin{proof}
If such a curve $g$ existed, it would be of the form $g = f_2$
where $f:\bbP^1\to\bbP^3$ would be a nonlinear contact curve
with ramification degrees $r_1(f)$ and $r_2(f)$.  Now
$$
5 = \deg(g) = 4 + r_1(f) + r_2(f),
$$
which, since $r_2(f)$ must be even, implies that $r_1(f) = 1$
and $r_2(f) = 0$.  Hence $\deg(f) = 3 + r_1(f) + \tfrac12r_2(f) = 4$.  
However, Proposition~\ref{prop: f_of_deg_4} 
shows that the only nonlinear contact curve~$f:\bbP^1\to\bbP^3$ 
of degree $4$ has $r_1(f) = 0$ and $r_2(f) = 2$.  

Thus, such a $g$ does not exist.
\end{proof}

\subsection{Degree $6$}
Now, I will classify the nonlinear rational null curves
of degree~$6$.

\begin{proposition}\label{prop: ratl-deg-6}
Up to projective equivalence, 
there are only two nonlinear null curves $g:\bbP^1\to\bbQ^3$
of degree~$6$.  One of these is unbranched,
and the other has two distinct branch points, each of order~$1$.
\end{proposition}

\begin{proof}
Let $g:\bbP^1\to\bbQ^3$ be a nonlinear null curve of degree~$6$
and let $f:\bbP^1\to\bbP^3$ be the Klein-corresponding nonlinear
contact curve.  From the formulae above,
$$
6 = \deg g = 4 + r_1(f) + r_2(f).
$$
Since $r_2(f)$ is even, there are two possibilities:
$\bigl(r_1(f),r_2(f)\bigr)$ is either $(0,2)$ or $(2,0)$.

First, if $\bigl(r_1(f),r_2(f)\bigr) = (0,2)$, 
then $\deg f = 3 + 0 + 1 = 4$,
and, by Proposition~\ref{prop: f_of_deg_4}, 
this $f$ is unique up to symplectic equivalence.  
In this case, $g = f_2$, since $R_1(g) = R_2(f) = p+q$
where $p,q\in\bbP^1$ are distinct, 
$g$ has two branch points of order~$1$.

Second, if $\bigl(r_1(f),r_2(f)\bigr) = (2,0)$
then $\deg f = 3 + 2 + 0 = 5$, and $R_1(f)=p+q$
where $p,q\in\bbP^1$ may be equal.

In the special case when $R_1(f) = 2{\cdot}p$, 
choose a meromorphic function $z$ on~$\bbP^1$ 
that has a single pole at $p$, and write
$$
f = \bigl[ \,v^0+z\,v^1+z^2\,v^2+z^3\,v^3+z^4\,v^4+z^5\,v^5\,\bigr].
$$
where $v^0,\ldots,v^5$ span $\bbC^4$ and $v^0$ and $v^5$ are not zero.  
The condition $R_1(f) = 2{\cdot}p$ implies that $v^3$ and $v^4$
are multiples of $v^5$, and so, by replacing $z$ by $z+c$
for an appropriate constant, it can be assumes that $v^4=0$, so that
$$
f = \bigl[ \,v^0 + z\,v^1 + z^2\,v^2 + (a\,z^3 + z^5)\, v^5\,\bigr].
$$
where $a$ is a constant.  Thus,
$$
\begin{aligned}
g = f_2 
&= \bigl[\,v^0{\w}v^1+2z\,v^0{\w}v^2 +z^2\,(3a\,v^0{\w}v^5+v^1{\w}v^2)
        +2az^3\,v^1{\w}v^5 \\
      &\qquad + 3z^4\,(5v^0{\w}v^5 + a v^2{\w}v^5)  
                +  4z^5\,v^1{\w}v^5 + 3z^6\,v^2{\w}v^5\,\bigr]
\end{aligned}              
$$
By inspection, whatever the value of~$a$, 
the seven coefficients of $z^k$ in this expression
span the entire $6$-dimensional space $\Lambda^2(\bbC^4)$.
Thus, $f$ is not a contact curve 
for any symplectic structure on~$\bbC^4$.

Supposing, instead, that $R_1(f) = p+q$, 
where $p,q\in\bbP^1$ are distinct,
let $z$ be a meromorphic function on~$\bbP^1$
with a simple pole at $p$ and a zero at~$q$.  
Then $f$ takes the form
$$
f = \bigl[ \,(1+az)\,v^0 + z^2\,v^2 + z^3\, v^3 + (bz^4+z^5)\,v^5\,\bigr].
$$
for some constants $a$ and $b$, 
where $(v^0,v^2,v^3,v^5)$ a basis of~$\bbC^4$.
Thus,
$$
\begin{aligned}
g = f_2 
&= \bigl[\,2v^0{\w}v^2+z\,(a\,v^0{\w}v^2+3\,v^0{\w}v^3) 
           +z^2\,(2a\,v^0{\w}v^3+4b\,v^0{\w}v^5)\\
&\qquad    +z^3\,\bigl((5+3ab)v^0{\w}v^5 + v^2{\w}v^3) 
           +z^4\,(4a\,v^0{\w}v^5 + 2b\, v^2{\w}v^5)  \\
 &\qquad  +z^5\,(3\,v^2{\w}v^5+ b\,v^3{\w}v^5 ) + 2z^6\,v^3{\w}v^5\,\bigr]
\end{aligned}              
$$
By inspection, whenever either $a$ or $b$ is nonzero,  
$g=f_2$ is linearly full in $\Lambda^2(\bbC^4)$, 
and, hence, $f$ is not contact 
for any symplectic structure on~$\bbC^4$. 

Meanwhile, if $a = b = 0$, then the formula for $g$ simplifies to
$$
g = \bigl[\,2\,v^0{\w}v^2+3z\,v^0{\w}v^3 
            +z^3\,\bigl(5\,v^0{\w}v^5 + v^2{\w}v^3) 
            +3z^5\,v^2{\w}v^5 + 2z^6\,v^3{\w}v^5\,\bigr]
$$
so that $g(\bbP^1)$ is linearly full in the $5$-dimensional 
subspace $W\subset\Lambda^2(\bbC^4)$ that is annihilated by
the symplectic form
$$
\beta = \xi_0\w\xi_5 - 5\,\xi_2\w\xi_3
$$
where $(\xi_0,\xi_2,\xi_3,\xi_5)$ is the basis of $(\bbC^4)^*$
 dual to the basis $(v^0,v^2,v^3,v^5)$ of $\bbC^4$.
Hence, $f$ is contact and $g:\bbP^1\to\bbQ^3\subset\bbP(W)$
is an unbranched (since $r_2(f) = 0$) null curve of degree~$6$.

This argument establishes the uniqueness up to projective
equivalence of such an $f:\bbP^1\to\bbP^3$ of degree~$5$ 
with $R_1(f)=p+q$ and $R_2(f)=0$ 
and hence the uniqueness up to equivalence 
of an unbranched null curve $g:\bbP^1\to\bbQ^3$ of degree~$6$. 
\end{proof}

\begin{remark}[Reducibility of a moduli space]
Note that the two corresponding rational contact curves are
\be
f = [1, z, z^3, z^4],
\ee
which has $g=f_2$ branched at $z=0$ and $z=\infty$, and
\be
f = [1, z^2, z^3, z^5]
\ee
which has $g = f_2$ unbranched, though $f$ itself is branched
at $z=0$ and $z=\infty$.

In each case, the projective subgroup $H\subset\SL(4,\bbC)$
that stabilizes $f$ has dimension $1$ (and has two components).
Consequently, the moduli space of such contact curves
for a given symplectic structure~$\beta$ is of the form
$\Sp(\beta)/H$ and hence has dimension~$9$.

Thus, the moduli space of nonlinear rational null curves in~$\bbQ^3$
of degree $6$ is disconnected.  Even when compactified using
geometric invariant theory, this moduli space will necessarily 
be reducible, being the union of two irreducible varieties 
of dimension~$9$. 
\end{remark}

\subsection{Degree $7$}
Finally, we treat the unbranched case in degree~$7$.

\begin{proposition}\label{prop: unbranched-ratl-deg-7}
There is no unbranched nonlinear null curve $g:\bbP^1\to\bbQ^3$
of degree~$7$.
\end{proposition}

\begin{proof}
Suppose that an unbranched nonlinear null curve~$g:\bbP^1\to\bbQ^3$
of degree~$7$ exists 
and let $f:\bbP^1\to\bbP^3$ be the Klein-corresponding contact curve.
By the formulae~\eqref{eq: unbranchedrationalnullcurves} 
of Example~\ref{ex: unbranchedrationalnullcurves}, it follows
that $f$ has degree~$6$ and satisfies $r_1(f) = 3$ and $r_2(f)=0$.
There are three cases to consider, depending on the structure of $R_1(f)$.

First, suppose that $R_1(f) = 3{\cdot}p$ for some $p\in\bbP^1$.
Choose a meromorphic $z$ on $\bbP^1$ with a simple pole at $p$
(and no other poles). Then $f$ takes the form
\be
f = [\,v^0+z\,v^1+z^2\,v^2+z^3\,v^3+z^4\,v^4+z^5\,v^5+z^6\,v^6\,]
\ee 
for some $v^0,\ldots,v^6\in\bbC^4$ with $v_0$ and $v_6$ nonzero.
Since $R_1(f) = 3{\cdot}p$, 
it follows that $v^3$, $v^4$, and $v^5$ 
are multiples of $v^6$.  By replacing $z$ by $z+c$ 
for some constant $c$, I can arrange that $v^5=0$,
so I do that.  Then $f$ takes the form
\be
f = [\,v^0 + z\,v^1 + z^2\,v^2 + (a\,z^3+b\,z^4+z^6)\,v^6\,]
\ee
for some constants $a$ and $b$, where $(v^0,v^1,v^2,v^6)$ are
a basis for $\bbC^4$. Then
\be
\begin{aligned}
g = f_2 &= [\,v^0{\w}v^1 + 2z\,v^0{\w}v^2 
     + z^2\,(3a\,v^0{\w}v^6+ v^1{\w}v^2)\\
     &\qquad + z^3\,(4b\,v^0+2a\,v^1){\w}v^6
             + z^4\,(3b\,v^1- a\,v^2){\w}v^6\\
     &\qquad + z^5\,(2b\,v^2- 6\,v^0){\w}v^6 + 5z^6\,v^1{\w}v^6
             + 4z^7\,v^2{\w}v^6].
\end{aligned}
\ee
By inspection the coefficients of the different powers of $z$
span $\Lambda^2(\bbC^4)$, no matter what the values of $a$ and $b$. 
Hence, this curve is linearly full in $\bbP(\Lambda^2(\bbC^4))$, 
and this $f$ is not contact for any symplectic structure on~$\bbC^4$.

Second, suppose that $R_1(f) = 2{\cdot}p + q$ for some $p,q\in\bbP^1$
that are distinct.
Choose a meromorphic $z$ on $\bbP^1$ with a simple pole at $p$
and a simple zero at $q$.  Then, because $R_1(f) = 2{\cdot}p + q$,
it follows that $f$ can be written in the form
\be
f = [\,(1+az)\,v^0 + z^2\,v^2 + z^3\,v^3 + (b\,z^4+c\,z^5+z^6)\,v^6\,]
\ee
for some constants $a$, $b$, and $c$ 
and vectors $v^0$, $v^2$, $v^3$, $v^6$ that form a basis of~$\bbC^4$.
Then
\be
\begin{aligned}
g = f_2 &= [\,2v^0{\w}v^2 + z\,v^0{\w}(a\,v^2+3\,v^3) 
     + z^2\,v^0{\w}(4b\,v^6+ 2a\,v^3)\\
     &\qquad + z^3\,((3ab{+}5c)\,v^0{\w}v^6+v^2{\w}v^3)
             + z^4\,((4ac{+}6)\,v^0+2b\,v^2){\w}v^6\\
     &\qquad + z^5\,(5a\,v^0+3c\,v^2 + b\,v^3){\w}v^6\\
     &\qquad + z^6\,(2c\,v^3+4\,v^2){\w}v^6
             + 3z^7\,v^3{\w}v^6].
\end{aligned}
\ee
Looking at the coefficients of the $0$-th, $1$-st, $7$-th, 
and $6$-th powers of $z$ in this formula, it follows that $f_2$
lies linearly fully in a space $W\subset\Lambda^2(\bbC^4)$ 
that contains $v^0{\w}v^2$, $v^0{\w}v^3$, $v^3{\w}v^6$, 
and $v^2{\w}v^6$, no matter what the values of $a$, $b$, and $c$.
The space $W$ must also contain the elements in the set
$$
\{\,a\,v^0{\w}v^6,\ b\,v^0{\w}v^6,\ (4ac+6)\,v^0{\w}v^6,\ 
  v^2{\w}v^3 + (3ab+5c)\,v^0{\w}v^6\}.
$$
No matter what the values of $a$, $b$, and $c$ are, the first three
elements will span the multiples of $v^0{\w}v^6$, and this,
combined with the fourth element, will force $W$ to contain $v^2{\w}v^3$
as well.  Thus $W = \Lambda^2(\bbC^4)$, 
implying that $f_2$ is nondegenerate, which is impossible if $f$
is to be a contact curve.  Thus, this case is also impossible.

Third, and finally, suppose that $R_1(f) = p+q+s$ where $p,q,s\in\bbP$
are distinct.  Let $z$ be the meromorphic function on $\bbP^1$
that has a pole at $p$, a zero at $q$ and satisfies $z(s) = 1$.  
(This uniquely specifies $z$.)  Then $f = [F(z)]$ where 
\be
F(z) = v^0 + z\,v^1 + z^2\,v^2 + z^3\,v^3 + z^4\,v^4 + z^5\,v^5 + z^6\,v^6 
\ee
and where $v^0,\ldots,v^6$ span~$\bbC^4$ and $v^0$ and $v^6$ are nonzero.
Moreover, because $p$ and $q$ are branch points of $f$, it follows
that $v^0{\w}v^1 = v^5{\w}v^6 = 0$, so that we must actually have
\be
F(z) = (1+az)\,v^0 + z^2\,v^2 + z^3\,v^3 + z^4\,v^4 + (bz^5+ z^6)\,v^6, 
\ee
for some constants $a$ and $b$.   Moreover, there can be only one
linear relation among the $5$ vectors $v^0,v^2,v^3,v^4,v^6$. 
Also, because $f$ must have degree~$6$, we cannot have $F(z_0)=0$
for any $z_0\in\bbC$, since then $F(z)/(z-z_0)$ would be a curve of
degree~$5$, forcing $f$ to have degree at most~$5$. 

Now, it turns out that it greatly simplifies the argument below 
to make a change of basis so that $F$ is written in the form
\be
\begin{aligned}
F(z) = (1+az)\,v^0 &+ z^2\,\bigl((v^2{-}v^3{-}(2a{+}3)v^0\bigr) \\
&+ z^3\,\bigl(2v^3{+}(a{+}2)v^0{+}(b{+}2)v^6\bigr) \\
&+ z^4\,\bigl(v^4{-}v^3{-}(2b{+}3)v^6\bigr) + (bz^5+ z^6)\,v^6, 
\end{aligned}
\ee
as can clearly be done.  The reason this is useful is that
the condition that $f$ have a branch point at $s$,
which is where $z = 1$, is equivalent
to the condition that $F(1)\w F'(1) = 0$, and computation now
shows that
$$
0 = F(1)\w F'(1) = 2\,v^2\w v^4.
$$ 
Hence $v^2$ and $v^4$ must be linearly dependent.  
Since there can only be one linear relation 
among the $5$ vectors $\{v^0,v^2,v^3,v^4,v^6\}$, it follows
that $v^2$ and $v^4$ must be multiples, not both zero, of a single
vector.  Consequently, after a renaming and a choice of two 
numbers $p$ and $q$, not both zero, we can write $F$ in the form
\be
\begin{aligned}
F(z) = (1+az)\,v^0 &+ z^2\,\bigl(p\,v^2{-}v^3{-}(2a{+}3)v^0\bigr) \\
&+ z^3\,\bigl(2v^3{+}(a{+}2)v^0{+}(b{+}2)v^6\bigr) \\
&+ z^4\,\bigl(q\,v^2{-}v^3{-}(2b{+}3)v^6\bigr) + (bz^5+ z^6)\,v^6, 
\end{aligned}
\ee 
where now $v^0,v^2,v^3,v^6$ are a basis of $\bbC^4$, and $a$, $b$,
$p$, and $q$ are constants, with $p$ and $q$ not both zero.

Now, let
$$
G(z) = \frac{F(z)\w F'(z)}{z(z-1)}
= G_0 + G_1\,z + G_2\,z^2 + \cdots + G_7\,z^7,
$$
so that $g = f_2 = [G(z)]$.  We must determine the conditions
on $a$, $b$, $p$, and $q$ in order that $g(\bbP^1)$ not lie linearly
fully in $\bbP(\Lambda^2(\bbC^4))$, which is that the eight
vectors $G_0,\ldots,G_7$ in $\Lambda^2(\bbC^4)$ should only span
a vector space of dimension at most $5$.

Let $B^1 = v^0{\w}v^2$, $B^2 = v^0{\w}v^3$, $B^3 = v^0{\w}v^6$,
$B^4 = v^2{\w}v^3$, $B^5 = v^2{\w}v^6$, and $B^6= v^3{\w}v^6$.
Then $B^1,\ldots,B^6$ form a basis of $\Lambda^2(\bbC^4)$. 
Thus, there is a $8$-by-$6$ matrix $M_{aj}$ 
such that $G_a = \sum_j\,M_{aj}\,B^j$.  
Calculation now yields that $M$ is
$$
\begin{pmatrix}
-2p&2&0&0&0&0\\
-p(a{+}2)&a{-}4&-3(b{+}2)&0&0&0\\
-(ap{+}2p{+}4q)&-3a&-2ab{-}4a{+}5b{+}6&0&0&0\\
-q(3a{+}4)&3a{+}4&6ab{+}9a{+}3b{+}12&-2p&-p(b{+}2)&b{+}2\\
q(a{+}2)&-(a{+}2)&-6ab{-}3a{-}9b{-}12&-2q&p(3b{+}4)&-(3b{+}4)\\
0&0&2ab{-}5a{+}4b{-}6&0&bq{+}4p{+}2q&3b\\
0&0&3(a{+}2)&0&q(b{+}2)&4{-}b\\
0&0&0&0&2q&-2
\end{pmatrix} 
$$

Now, in order that $g$ not be branched at $z=0$, we must have $G_0$
and $G_1$ linearly independent, i.e., the first two rows of $M$ must
be of rank $2$ and inspection shows that this requires that at least
one of $p$ and $b{+}2$ must be nonzero.  Similarly, because $g$
is not branched at $z=\infty$, at least one of $q$ and $a{+}2$
must be nonzero. 

In order for the rank of $M$ to be at most $5$, all of the 
$6$-by-$6$ minors of $M$ must be zero.  By computation,
the determinant of the first 6 rows is $-48\bigl( (3p{+}q)b+4p+2q\bigr)^3$
while the determinant of the last 6 rows 
is  $-48\bigl( (p{+}3q)a+2p+4q\bigr)^3$.  Thus, we must have
$$
(3p{+}q)b+4p+2q = (p{+}3q)a+2p+4q = 0.
$$
Recall that $p$ and $q$ cannot simultaneously vanish.  It is now
apparent that $3p{+}q$ cannot be zero either, since the above
equations would then imply that $4p{+}2q = 0$, forcing $p=q=0$,
which cannot happen.  Similarly, $p{+}3q$ cannot be zero.  Thus,
we can solve for $a$ and $b$ in the form
$$
a = -\frac{2p{+}4q}{p{+}3q}\qquad\text{and}\qquad
b = -\frac{4p{+}2q}{3p{+}q}\,.
$$
From these formulae, we can see that, if $q$ were zero, then $a$
would be $-2$, but $q = a{+}2 = 0$ is not allowed.  Hence $q$ is
non-zero.  Similarly $p$ must be nonzero.

Finally, computing the determinant of the $6$-by-$6$ minor of~$M$
obtained by deleting the third and sixth rows of~$M$ yields
$$
\frac{8640\,pq\,(p+q)^3}{(p + 3q)(3p + q)}\,.
$$
Consequently, since $p$ and $q$ cannot be zero, 
it must be that $p+q=0$, which implies that $a = -1$ and $b=-1$.
Further, by scaling $v^2$, we can arrange that $p=1$ and $q=-1$.
Thus, the only possibility for $f = [F]$ is to have
\be
\begin{aligned}
F(z) = (1{-}z)\,v^0 &+ z^2\,\bigl(v^2{-}v^3{-}v^0\bigr) \\
&+ z^3\,\bigl(2v^3{+}v^0{+}v^6\bigr) \\
&- z^4\,\bigl(v^2{+}v^3{+}v^6\bigr) + (z^6{-}z^5)\,v^6, 
\end{aligned}
\ee 
However, since $F(1) = 0$, it follows that $f = [F(z)/(z{-}1)]$
can only have degree $5$ at most.  This contradiction
shows that the desired $f$ does not exist.

Hence, there is no unbranched null curve $g:\bbP^1\to\bbQ^3$
of degree~$7$, as claimed.
\end{proof}

\begin{remark}
There does exist a \emph{branched}
nonlinear null curve $g:\bbP^1\to\bbQ^3$ of degree~$7$.  
The corresponding contact curve $f:\bbP^1\to\bbP^3$ 
is of the form
$$
f = [\,(1{-}5z^2)\,v^0+(z{-}3z^2)\,v^1+(z^4{-}3z^3)\,v^4+(z^5{-}5z^3)\,v^5\,]
$$
for a meromorphic parameter $z$ on~$\bbP^1$ and a basis~$v^0,v^1,v^4,v^5$
of~$\bbC^4$.  This $f$ satisfies $R_1(f) = s$,
where $z(s)=1$, and $R_2(f) = p+q$, where $p$ is the pole of $z$ and $q$
is the zero of $z$.  

It can be shown~\cite{BandW1} that, up to projective equivalence,
this is the unique contact curve~$f:\bbP^1\to\bbP^3$ 
with $r_1(f)=1$ and $r_2(f)=2$.

Since any nonlinear null curve~$g:\bbP^1\to\bbQ^3$ 
of degree~$7$ must satisfy $r_1(f)+r_2(f) = 7-4 = 3$
and since $r_2(f)$ must be even, it follows that such a curve,
which must be branched by Proposition~\ref{prop: unbranched-ratl-deg-7}, 
must have $r_1(f) = 1$ and $r_2(f)=2$. 

Hence all of the nonlinear rational null curves~$g:\bbP^1\to\bbQ^3$
of degree~$7$ form a single $\Sp(2,\bbC)$-orbit of dimension~$10$.
\end{remark}

\bibliographystyle{hamsplain}

\providecommand{\bysame}{\leavevmode\hbox to3em{\hrulefill}\thinspace}

\end{document}